\documentclass{article}

\usepackage[dvipsnames]{xcolor}
\usepackage{graphicx} 
\usepackage[margin=1.2in]{geometry}
\usepackage{amsmath,amssymb,amsthm}
\usepackage{subcaption}
\usepackage[colorlinks=true, linkcolor=blue, citecolor=red]{hyperref}
\usepackage{cleveref}  % must be added after hyperref
\usepackage{tikz}
\usepackage{xcolor}
\usepackage{pifont}
\usepackage{authblk}

\definecolor{myBlue}{rgb}{0.25, 0.0, 1.0}
\definecolor{myLightBlue}{rgb}{0.39, 0.58, 0.93}
\colorlet{myGreen}{green!50!black}
\colorlet{myLightgreen}{green}
\definecolor{AppleGreen}{rgb}{0.55, 0.71, 0.0}
\theoremstyle{plain}
\newtheorem{thm}{Theorem}[section]
\newtheorem{lem}[thm]{Lemma}
\newtheorem{fact}[thm]{Fact}

\newtheorem{conj}[thm]{Conjecture}

\newtheorem{cor}[thm]{Corollary}

\theoremstyle{definition}

\newcommand{\surface}{\mathbb{S}}

\newcommand{\defn}[1]{\textcolor{purple}{\emph{#1}}}
\DeclareMathOperator{\cl}{cl}
\DeclareMathOperator{\bd}{bd}

\title{Triangulated spheres with holes in triangulated surfaces}
\author[a]{Katie Clinch}
\affil[a]{School of Mathematics and Physics, University of Queensland, Australia}
\author[b]{Sean Dewar}
\affil[b]{School of Mathematics, University of Bristol, UK}
\author[c]{Niloufar Fuladi}
\affil[c]{LORIA, CNRS, INRIA, Université de Lorraine, Nancy, France}
\author[d]{Maximilian Gorsky}
\affil[d]{Discrete Mathematics Group, Institute for Basic Science, Daejeon, South Korea}
\author[d]{Tony Huynh}
\author[e]{Eleftherios Kastis}
\affil[e]{School of Mathematical Sciences, Lancaster University, UK}
\author[f]{Atsuhiro Nakamoto}
\affil[f]{Faculty of Environment and Information Sciences, Yokohama National University, Japan}
\author[e]{Anthony Nixon}
\author[g]{Brigitte Servatius}
\affil[g]{Mathematical Sciences, Worcester Polytechnic Institute, USA}
\date{}

\begin{document}

\maketitle
\begin{abstract}
Let $\surface_h$ denote a sphere with $h$ holes.  
Given a triangulation $G$ of a surface $\mathbb{M}$, we consider the question of when $G$ contains a spanning subgraph $H$ such that $H$ is a triangulated $\surface_h$. We give a new short proof of a theorem of Nevo and Tarabykin~\cite{NT} that every triangulation $G$ of the torus contains a spanning subgraph which is a triangulated cylinder.  For arbitrary surfaces, we prove that every high facewidth triangulation of a surface with $h$ handles contains a spanning subgraph which is a triangulated $\surface_{2h}$.   We also prove that for every $0 \leq g' < g$ and $w \in \mathbb{N}$, there exists a triangulation of facewidth at least $w$ of a surface of Euler genus $g$ that does not have a spanning subgraph which is a triangulated $\surface_{g'}$.  Our results are motivated by, and have applications 
 for, rigidity questions in the plane.  
\end{abstract}

%============================================================================
\section{Introduction}
%============================================================================

A \defn{sphere with $h$ holes} is the surface $\surface_{h}$ obtained from the sphere by removing $h$ open disks, whose closures are pairwise disjoint. The main question we address in this paper is when a triangulation of a surface contains a spanning subgraph which is a triangulated $\surface_{h}$. Before stating our main results we review some basic notions from topological graph theory.  For more details, we refer the reader to~\cite{moharthom}. 

For each $h \in \mathbb{N}$, we let $\mathbb{T}_h$ denote the orientable surface with $h$ handles and we let $\mathbb{N}_h$ denote the nonorientable surface with $h$ crosscaps. 
The \defn{Euler genus} of $\mathbb{T}_h$ and $\mathbb{N}_h$ is $2h$ and $h$, respectively. Recall that a connected compact topological space $\mathbb{M}$ is a \defn{surface (with boundary)} if every point in $\mathbb{M}$ has an open neighborhood homeomorphic to $\mathbb{D}:=\{(x,y) \in \mathbb{R}^2 : x^2+y^2 < 1\}$ or to $\mathbb{H}:=\{(x,y) \in \mathbb{D} : y \geq 0\}$.  The set of points of $\mathbb{M}$ with an open neighborhood homeomorphic to $\mathbb{H}$ is called the \defn{boundary} of $\mathbb{M}$. 
% \katie{with equality?} \tonyh{I think it is correct as written.  Every point has an open neighbourhood which is either a `disk' or a `half-disk', and the set of points with a `half-disk' neighbourhood are precisely the boundary of the surface.} \katie{OK. I misunderstood what was meant by ``the latter condition'' as being about the $y$ inequality. Suggest replacing this with something explicit: ``the set of points whose intersection gives the half-plane''? Or perhaps some Oxford commas, or labeling of conditions as (i), (ii) to make this clearer} 
% \tonyh{I rewrote it.  Let me know if it is clear now.}
If a surface has an empty boundary, it is said to be \defn{closed}.
Due to compactness, the boundary of any surface is a collection of disjoint circles.
We let $\overline{\mathbb{M}}$ be the closed surface obtained from $\mathbb{M}$ by gluing a disk onto each boundary component. By the classification theorem of surfaces, every closed surface is homeomorphic to $\mathbb{T}_h$ or $\mathbb{N}_h$, for some $h$.  

For $X \subseteq \mathbb{M}$, we let $\bd(X)$ and $\cl(X)$ denote the (topological) boundary and closure of $X$, respectively. 
A \defn{simple arc} $P$ in $\mathbb{M}$ is the image of a continuous and injective map $I:[0,1] \to \mathbb{M}$.
The \defn{ends} of $P$ are the points $I(0)$ and $I(1)$.  
A \defn{circle} $C$ is the image of a continuous and injective map $I:S^1 \to \mathbb{M}$, where $S^1$ is the 1-dimensional sphere.  A circle $C$ in $\mathbb{M}$ is \defn{contractible} if some component $D$ of $\mathbb{M} \setminus C$ is a disk with $\bd(D)=C$.  Otherwise, $C$ is \defn{noncontractible}.
Two disjoint noncontractible circles $C$ and $D$ are \defn{homotopic} in $\mathbb{M}$ if some component $\Sigma$ of $\mathbb{M} \setminus (C \cup D)$ is an open cylinder with $\bd(\Sigma)=C \cup D$.  Note that this is equivalent to the usual definition of homotopy in topology (see~\cite{levine63}).

All graphs in this paper are finite and simple.  An \defn{embedding} of a graph $G$ on a surface $\mathbb{M}$ is a function $\psi$ with domain $V(G) \cup E(G)$, such that 
\begin{itemize}
    \item $\psi(v)$ is a point of $\mathbb{M}$ for each $v \in V(G)$,
    \item $\psi(u) \neq \psi(v)$ for all distinct $u,v\in V(G)$,
    \item $\psi(uv)$ is a simple arc in $\mathbb{M}$ with ends $\psi(u)$ and $\psi(v)$ for all $uv\in E(G)$, and
    \item $\psi(uv) \cap \psi(wx) \subset \{\psi(u), \psi(v)\}$, for all distinct $uv, wx \in E(G)$. %KC - if uv = wx then the intersection is the whole arc
\end{itemize}

The resulting \defn{embedded graph} is denoted by the pair $(G,\psi)$.  For every subgraph $H$ of $G$, we let $\psi(H)$ be the subset of $\mathbb{M}$ corresponding to $H$. To be precise, $\psi(H):=\bigcup_{v \in V(H)}\{\psi(v)\} \cup \bigcup_{xy \in E(H)} \psi(xy)$. A \defn{face} of $(G,\psi)$ is a connected component of $\mathbb{M} \setminus \psi(G)$. We only consider embeddings where every face is homeomorphic to an open disk. 

 We say $(G,\psi)$ is a \defn{triangulation} of $\mathbb{M}$, or equivalently $(G,\psi)$ is a \defn{triangulated} $\mathbb{M}$,  if $\bd(\mathbb{M}) \subseteq \psi(G)$ and the boundary of every face of $(G,\psi)$ corresponds to a copy of $K_3$ in $G$.  
 
 If $\mathbb{M}$ is a (closed) surface, the \defn{facewidth} of $(G,\psi)$ is the minimum number of intersections between $\psi(G)$  and any noncontractible circle in $\mathbb{M}$ \cite[Section 5.5]{moharthom}.
  Nakamoto and Nozawa~\cite{NN15}, and Nevo and Tarabykin~\cite{NT}, independently proved that every triangulation of the projective plane ($\mathbb{N}_1$) contains a spanning subgraph which is a triangulated disk ($\mathbb{S}_1$).  

\begin{thm}[\cite{NN15} and~\cite{NT}] \label{thm:projectiveplane}
        Every triangulation of the projective plane contains a spanning subgraph which is a triangulated disk.  
\end{thm}

Nevo and Tarabykin~\cite{NT} also proved that every triangulation of the torus ($\mathbb{T}_1$) contains a spanning subgraph which is a triangulated cylinder ($\mathbb{S}_2$).

\begin{thm}[\cite{NT}] \label{thm:toruscylinder}
        Every triangulation of the torus contains a spanning subgraph which is a triangulated cylinder.
\end{thm}

We propose the following generalization of the above two results to arbitrary surfaces. 
\begin{conj} \label{conj:main}
    For all $g \geq 0$, every triangulation of a surface of Euler genus $g$ contains a spanning subgraph which is a triangulated $\surface_g$.   
\end{conj}
 
We remark that \Cref{conj:main} is a strengthening of a conjecture of Nevo and Tarabykin~\cite{NT}, who conjectured that every triangulation of a surface contains a spanning subgraph which is planar and Laman (precise definitions will be given later).

Our main results are as follows. We begin by reproving~\Cref{thm:toruscylinder}.
The proof of~\Cref{thm:toruscylinder} in~\cite{NT} uses the set of irreducible triangulations\footnote{A triangulation of a surface is \defn{irreducible} if there is no edge whose contraction produces another triangulation of the surface.} of the torus.  Barnette and Edelson~\cite{BD88, BD89} proved that the set of irreducible triangulations of a surface is always finite, but the complete list is only known for surfaces of Euler genus at most 4 \cite{Sulanke06}. In contrast to the approach in~\cite{NT}, our proof of~\Cref{thm:toruscylinder} is very short and avoids using the list of irreducible triangulations of the torus.  

Next, we prove~\Cref{conj:main} for high facewidth triangulations of orientable surfaces. Recall that $(2h-1)!!=1\cdot3 \cdots (2h-1)$, where $(-1)!!=0!!=1$.  

\begin{thm} \label{thm:highfacewidth}
    There exists a function $\gamma(h) \in O(h \cdot 2^h \cdot (2h-1)!!)$ such that the following holds.  
    Let $(G,\psi)$ be a triangulation of $\mathbb{T}_h$ such that $(G,\psi)$ has facewidth at least $\gamma(h)$.
    Then $G$ contains a spanning subgraph $H$ such that $(H,\psi|_H)$ is a triangulated $\surface_{2h}$.
\end{thm}

More specifically, the closed form for our function $\gamma(h)$ is 
\begin{equation}\label{eq:gamma}
    \gamma(h) = (4h+1) \cdot 2^{h-1} \cdot (2h - 3)!! + \sum_{i=0}^{h-2} \left( (20i +11) \cdot 2^{i} \cdot (2i-1)!! \right).
\end{equation}
The function $\gamma(h)$ is bounded above by $(h \cdot (4h+1) \cdot 2^{h-1} \cdot (2h - 3)!!)$, which is asymptotically equivalent to $(h \cdot 2^{h} \cdot (2h - 1)!!)$.

Finally, we prove that the number of holes in~\Cref{thm:highfacewidth} and \Cref{conj:main} cannot be reduced.  

\begin{thm} \label{thm:fewerholes}
     For every $g > g' \geq 0$ and $w \in \mathbb{N}$, there exists a triangulation $(G,\psi)$ of a surface of Euler genus $g$ with facewidth at least $w$ such that $(H,\psi|_H)$ is not a triangulated $\surface_{g'}$ for any spanning subgraph $H$ of $G$.  
\end{thm}

\subsection{Motivation} \label{sec:motivation}

The initial motivation for our results comes from rigidity theory, although we believe they are of independent interest. Our proofs do not use any rigidity theory.

In rigidity theory, a \defn{framework} $(G,p)$ is an ordered pair consisting of a graph $G$ and a map $p:V(G)\rightarrow \mathbb{R}^d$ assigning positions in Euclidean $d$-space to the vertices of $G$. The framework is \defn{$d$-rigid} if every edge-length-preserving continuous motion of the vertices arises from an isometry of $\mathbb{R}^d$. Considering $p$ as a vector in $\mathbb{R}^{d|V(G)|}$ we say that $(G,p)$ is \defn{generic} if the coordinates of $p$ form an algebraically independent set over $\mathbb{Q}$. It is well known \cite{AsimowRoth} that $d$-rigidity is a generic property; that is either every generic framework of $G$ is $d$-rigid or every generic framework of $G$ is not $d$-rigid. Hence we call a graph $G$ \defn{$d$-rigid} if there exists some generic framework $(G,p)$ in $\mathbb{R}^d$ that is $d$-rigid. Moreover, we say that $G$ is \defn{minimally $d$-rigid} if $G$ is $d$-rigid but $G-e$ is not $d$-rigid for every $e\in E(G)$. 

A closely related notion is infinitesimal rigidity. 
An \defn{infinitesimal motion} of $(G, p)$ is a map $\dot
p:V(G)\to \mathbb{R}^d$ satisfying the system of linear equations:
\begin{equation*}
(p(v)-p(w))\cdot (\dot p(v)-\dot p(w)) = 0 \qquad \mbox{ for all $vw \in E(G)$}.
\end{equation*}
The framework $(G,p)$ is \defn{infinitesimally $d$-rigid} if the only infinitesimal motions arise from isometries of $\mathbb{R}^d$. Infinitesimal $d$-rigidity is a sufficient condition for $d$-rigidity that is easier to work with, and in the generic case the two notions coincide \cite{AsimowRoth}.

When $d=1$ it is easy to prove that $G$ is $1$-rigid if and only if it is connected. When $d=2$, Pollaczek-Geiringer proved that $G$ is minimally 2-rigid if and only if $G$ is a Laman graph \cite{L70,PG27}. Here $G$ is a \defn{Laman graph} if $|E(G)|=2|V(G)|-3$ and every subgraph $G'$ of $G$ with at least one edge satisfies $|E(G')| \leq 2|V(G')|-3$. When $d>2$ such a characterisation is a central open problem but the case of triangulations of surfaces is well known, see~\cite{Fogel}.

The following natural question was first posed by Adiprasito and Nevo \cite{AN}, see also a very similar question in earlier work of Fekete and Jord\'an \cite{FJ05}. For a $d$-rigid graph $G$, how small can a subset $S \subseteq \mathbb{R}^d$ be, so that there is an infinitesimally rigid framework $(G,p)$ with $p(v) \in S$ for every vertex $v \in V(G)$?
We define $c_d(G)$ to be the smallest possible cardinality of such a set $S$,
and for a family of graphs $\mathcal{G}$, we set $c_d(\mathcal{G}) := \sup \{ c_d(G) :  G \in \mathcal{G} \}$.

It follows from observations of Fekete and Jord\'an that, if $\mathcal{G}_d$ is the family of $d$-rigid graphs, then $c_1(\mathcal{G}_1) = 2$ and $c_d(\mathcal{G}_d) = \infty$ for $d \geq 2$; see \cite[Section 4]{FJ05}, since in this case there are $d$-rigid graphs without any proper $d$-rigid subgraphs.
In contrast to this,
Kir\'{a}ly \cite{Kir} proved that,
if $\mathcal{T}_h$ and $\mathcal{N}_h$ are the family of triangulations of the surfaces $\mathbb{T}_h$ and $\mathbb{N}_h$ respectively, then $c_2(\mathcal{T}_h) = O(\sqrt{h})$ and $c_2(\mathcal{N}_h) = O(\sqrt{h})$,
and if $\mathcal{P}$ is the family of planar Laman graphs, then $c_2(\mathcal{P}) \leq 26$. 
Using Kir\'{a}ly's latter result in conjunction with the existence of triangulated disks and other such 2-rigid planar graphs,
Nevo and Tarabykin, in \cite{NT}, proved that $c_2(\mathcal{T}_1) \leq 26$ and $c_2(\mathcal{N}_h) \leq 26$ for $h \leq 2$. Kaszanitzky \cite{VK23} has recently extended these results  by showing that $c_2(\mathcal{T}_2) \leq 26$.

Since a triangulated $\surface_g$ is a 2-rigid planar graph \cite[Lemma 6]{VK23},
\Cref{thm:highfacewidth} implies the following analogous result for triangulations of arbitrary orientable surfaces with sufficiently large facewidth.

\begin{cor}
    Let $\gamma(h)$ be as in \eqref{eq:gamma}.
    Then $c_2(G) \leq 26$ for every triangulation $(G, \psi)$ of $\mathbb{T}_h$ of facewidth at least $\gamma(h)$.
\end{cor}

%============================================================================
\section{The torus}
%====================================================================

In this section, we reprove~\Cref{thm:toruscylinder}.  Let $(G, \psi)$
be an embedded graph and $C$ and $D$ be vertex-disjoint cycles of $G$.  We say that $C$ is
\defn{noncontractible} if $\psi(C)$ is a noncontractible circle in $\mathbb{M}$.  We say that $C$ and $D$ are \defn{homotopic} if $\psi(C)$ and $\psi(D)$ are homotopic.

We require the following theorem of Schrijver~\cite{schrijver93}. 

\begin{thm}[\cite{schrijver93}] \label{thm:torusfacewidth}
Every graph embedded on the torus with facewidth at least $t$ contains $\lfloor \frac{3t}{4} \rfloor$ vertex-disjoint noncontractible cycles.  
\end{thm}

We also require the following well-known fact; a proof is included for completeness.

\begin{fact}\label{fact:cycleshomotopic}
    Any two noncontractible circles in the torus that are disjoint are also homotopic.
\end{fact}

\begin{proof}
Let $C_1$ and $C_2$ be disjoint noncontractible circles on the torus $\mathbb{T}_1$.  Since $C_2$ is disjoint from $C_1$, by compactness, $C_2$ is also disjoint from a sufficiently small neighbourhood of $C_1$.  In other words, there is a (closed) cylinder $\Sigma_1 \subseteq \mathbb{T}_1$ such that $C_1 \subseteq \Sigma_1 \setminus \bd(\Sigma_1)$ and $C_2 \cap \Sigma_1=\emptyset$.  Since $C_1$ is noncontractible, $\Sigma_2:=\mathbb{T}_1 \setminus \Sigma_1$ is an (open) cylinder.  Let $D_1$ and $D_2$ be the boundary components of $\Sigma_1$ (and thus also the boundary components of $\cl(\Sigma_2$)).  By construction, $C_1$ is homotopic to $D_1$.   
Since $C_2\subseteq \Sigma_2$ is noncontractible, we conclude that $\Sigma_2 \setminus C_2$ has two components $\Sigma_2^1$ and $\Sigma_2^2$ each of which is an (open) cylinder.  Moreover, for each $i \in [2]$, the boundary components of $\cl(\Sigma_2^i)$ are $C_2$ and $D_i$.  Hence, $C_2$ is homotopic to $D_1$ (and hence also to $C_1$). 
\end{proof}

\begin{proof}[Proof of~\Cref{thm:toruscylinder}]
Let $(G, \psi)$ be a triangulation of the torus $\mathbb{T}_1$.  Since $(G, \psi)$ has facewidth at least 3,  $G$ contains two vertex-disjoint noncontractible cycles $C_1$ and $C_2$ by~\Cref{thm:torusfacewidth}. By~\Cref{fact:cycleshomotopic}, $\psi(C_1), \psi(C_2)$ are homotopic. Thus, $\psi(C_1) \cup \psi(C_2)$ bounds a cylinder $\Sigma(C_1, C_2)$ in $\mathbb{T}_1$.

We choose $C_1$ and $C_2$ such that $\Sigma(C_1, C_2)$ is inclusionwise maximal in $\mathbb{T}_1$. 
We claim that $\Sigma(C_1, C_2)$ contains all vertices of $G$. If not, then there exists a triangle $T=xyz$ of $G$ such that $\psi(T)$ is the boundary of a face of $(G, \psi)$, with $xy \in E(C_2)$ and $\psi(z) \notin \Sigma(C_1, C_2)$. Thus, $C_1$ and $(C_2 \cup \{xz, yz\}) \setminus \{xy\}$ contradicts the choice of $C_1$ and $C_2$.  
\end{proof}

%============================================================================
\section{The high facewidth case}
%============================================================================

In this section, we prove~\Cref{thm:highfacewidth}.  Let $(G, \psi)$ be a graph embedded on a surface $\mathbb{M}$ and $C$ be a cycle of $G$.  We say that $C$ is \defn{surface nonseparating} if the set $\mathbb{M}\setminus\psi(C)$ is (topologically) connected.  A key ingredient is the following theorem of Brunet, Mohar, and Richter~\cite{BMR96}. 

\begin{thm}[{\cite[Corollary~7.2]{BMR96}}]\label{thm:nestedcycles}
Let $h \geq 1$, $t \geq 3$, and $(G,\psi)$ be a graph embedded on $\mathbb{T}_h$ with facewidth at least $t$. Then $(G,\psi)$ contains at least $\lfloor \frac{t-1}{2} \rfloor$ vertex-disjoint surface nonseparating cycles that are  pairwise homotopic. 
\end{thm} 

For all $h,q$ we let $\mathbb{T}_{h,q}$ denote the surface with $h$ handles and $q$ holes.  We begin by proving the following lemma.

\begin{lem} \label{lem:homtopiccyles}
    There exists a computable function $\phi(h,q)$ such that the following holds.  Let $(G, \psi)$ be a graph embedded on $\mathbb{T}_{h,q}$ such that $h \geq 1$ and $(G,\psi)$ has facewidth at least $\phi(h,q)$ on $\overline{\mathbb{T}_{h,q}}$.   
    Then $(G,\psi)$ contains vertex-disjoint surface nonseparating cycles $C_1, C_1', \dots, C_h, C_h'$ such that $C_i$ is homotopic to $C_i'$ on $\mathbb{T}_{h,q}$ for all $i \in [h]$, and $C_i$ is not homotopic to $C_j$ on $\mathbb{T}_{h,q}$ for all distinct $i,j \in [h]$. 
\end{lem}

\begin{proof}
    We define $\phi(h,q)$ recursively as follows.  
    Let $\phi(1,q)=2q+5$ for all $q \geq 0$ and $\phi(h,q)=(2q+2)(\phi(h-1, q+2)+5)+1$ for all $q \geq 0$ and $h \geq 2$.
    We proceed by induction on $h$. 
    
    If $h=1$, then by applying~\Cref{thm:torusfacewidth} in $\overline{\mathbb{T}_{1,q}}$, we conclude that $(G,\psi)$ contains a collection of $q+2$ vertex-disjoint surface nonseparating cycles on $\overline{\mathbb{T}_{1,q}}$.  By~\Cref{fact:cycleshomotopic},
    these $q+2$ cycles are pairwise homotopic on $\overline{\mathbb{T}_{1,q}}$.   Order these $q+2$ cycles as $D_1, \dots, D_{q+2}$ such that for all $1 \leq i < j< k \leq q+2$, $\psi(D_i \cup D_k)$ bounds an open cylinder $\Sigma(D_i,D_k)$ on $\overline{\mathbb{T}_{1,q}}$ containing $\psi(D_j)$.
    Observe that the open cylinders $\Sigma(D_1,D_{2}), \ldots ,\Sigma(D_{q+1},D_{q+2})$ are pairwise disjoint.  Let $O_1, \dots , O_q$ be the boundaries of the holes of $\mathbb{T}_{1,q}$.  For each $i \in [q]$, let $\Delta(O_i)$ be the open disk in $\overline{\mathbb{T}_{1,q}}$ bounded by $O_i$. Since $\Delta(O_i)$ is disjoint from $\psi(G)$, we conclude that 
   $\Delta(O_i)$ is either contained within exactly one open cylinder $\Sigma(D_j,D_{j+1})$, or $\Delta(O_i)$ is disjoint from every open cylinder $\Sigma(D_1,D_2), \ldots , \Sigma(D_{q+1},D_{q+2})$.
   By the pigeonhole principle, there must be at least one index $a \in [q+1]$ such that $\Sigma(D_a, D_{a+1})$ is disjoint from $\bigcup_{i \in [h]}\Delta(O_i)$. Thus, $D_a$ and $D_{a+1}$ are homotopic in $\mathbb{T}_{1,q}$.  So, we may take $C_1=D_a$ and $C_1'=D_{a+1}$.  
    
    Thus, we may assume $h \geq 2$.  By~\Cref{thm:nestedcycles}, $(G,\psi)$ contains  $m:=\frac{\phi(h,q)-1}{2}$ vertex-disjoint surface nonseparating cycles that are pairwise homotopic on $\overline{\mathbb{T}_{h,q}}$.  Again, order these $m$ cycles as $D_1, \dots, D_m$ such that for all $1 \leq i < j< k \leq m$, $\psi(D_i \cup D_k)$ bounds an open cylinder $\Sigma(D_i,D_k)$ on $\overline{\mathbb{T}_{h,q}}$ containing $\psi(D_j)$. 
     Let $\ell:=\phi(h-1, q+2)+4$. By the definition of $\phi$, we have $m = (q+1)(\ell+1)$.
     For each $i \in [m-\ell+1]$ define $\Sigma_i:=\cl(\Sigma(D_{i}, D_{i+\ell-1}))$ and define $i$ to be \defn{unholy} if $\Sigma_i$ is disjoint from the boundary of each hole of $\mathbb{T}_{h,q}$. 
   Let $I:=\{1+(\ell+1)(i-1) : i \in [q+1]\}$.   Let $O$ be the boundary of an arbitrary hole of $\mathbb{T}_{h,q}$ and $\Delta(O)$ be the open disk in $\overline{\mathbb{T}_{h,q}}$ bounded by $O$. Since $\Delta(O)$ is disjoint from $\psi(G)$, we conclude that 
  $O$ intersects at most two of $\psi(D_1), \dots, \psi(D_m)$. 
  We claim that there is at most one $i \in I$ such that $O \cap \Sigma_i \neq \emptyset$.  Suppose not.  Let $i<j$ be  elements of $I$ such that $O$ intersects both $\Sigma_i$ and $\Sigma_j$.  Then either $O$ intersects $\psi(D_k)$ for all $k \in [i+\ell-1, j]$ or $O$ intersects $\psi(D_k)$ for all $k \in [1, i] \cup [j+\ell-1, m]$.  Since both these sets have size at least $3$, we conclude that $O$ intersects at least three of  $\psi(D_1), \dots, \psi(D_m)$, which is a contradiction.   Since $\mathbb{T}_{h,q}$ has only $q$ holes and $|I|=q+1$, it follows that at least one index $u \in I$ is unholy.  

    Define $D_i':=D_{i+u-1}$ for all $i \in [\ell]$.  By the definition of unholy, $D_1', \dots, D_\ell'$ are pairwise homotopic on $\mathbb{T}_{h,q}$ and each $\psi(D_i')$ is disjoint from $\bd(\mathbb{T}_{h,q})$.
    Let $a:=\lceil \frac{\phi(h-1, q+2)}{2} \rceil$, $b:=a+3$, $\mathbb{M}:=\mathbb{T}_{h,q} \setminus \Sigma(D_a',D_b')$, and $G':=\psi^{-1}(\psi(G) \cap \mathbb{M})$. Since $\psi(D_a' \cup D_b')$ is disjoint from $\bd(\mathbb{T}_{h, q})$, we conclude that $\mathbb{M}$ is homeomorphic to $\mathbb{T}_{h-1, q+2}$. 
    Moreover, every noncontractible circle $C$ on $\overline{\mathbb{M}}$ is either noncontractible on $\overline{\mathbb{T}_{h,q}}$, 
      or it crosses one of the faces bounded by $\psi(D_a')$ or $\psi(D_b')$. 
      If it crosses the face bounded by $\psi(D_a')$, then since $\psi(D_i')$ is homotopic to $\psi(D_a')$ for all  $i\in [a]$,
      it follows that $C$ intersects $\psi(D_i')$ at least twice for each $i \in [a]$. Similarly, if $C$ crosses the face bounded by $\psi(D_b')$ then $C$ intersects $\psi(D_i')$ at least twice for all $i \in \{b, b+1, \dots, \ell\}$. 
     We conclude that the facewidth of $(G',\psi|_{G'})$ on $\overline{\mathbb{M}}$ is at least $\min \{\phi(h,q), 2a, 2(\ell-b+1)\} \geq \phi(h-1, q+2)$. 
     
     By induction, the embedded graph $(G', \psi|_{G'})$ contains vertex-disjoint surface nonseparating cycles $C_1, C_1', \dots, C_{h-1}, C_{h-1}'$ such that $C_i$ is homotopic to $C_i'$ (on $\mathbb{M}$) for all $i \in [h-1]$, and $C_i$ is not homotopic to $C_j$ (on $\mathbb{M}$) for all distinct $i,j \in [h-1]$. 
     Setting $C_h=D_{a+1}'$ and $C_{h}'=D_{a+2}'$ we have that $C_1, C_1', \dots , C_h, C_h'$ are the required cycles in $(G,\psi)$. 
\end{proof}

We are now ready to prove~\Cref{thm:highfacewidth}.

\begin{proof}[Proof of~\Cref{thm:highfacewidth}]  
For all $h \geq 1$, define $\gamma(h)=\phi(h,0)$, where $\phi$ is the function from~\Cref{lem:homtopiccyles}.  The closed form of $\gamma(h)$ is the function described in \eqref{eq:gamma}.  
By~\Cref{lem:homtopiccyles}, $G$ contains vertex-disjoint surface nonseparating cycles $C_1, C_1', \dots, C_h, C_h'$ such that $C_i$ is homotopic to $C_i'$ for all $i \in [h]$, and $C_i$ is not homotopic to $C_j$ for all distinct $i,j \in [h]$.  For each $i \in [h]$, let $\Sigma(C_i, C_i')$ be an (open) cylinder on $\mathbb{M}$ bounded by $\psi(C_i \cup C_i')$.  Since $C_i$ and $C_j$ are vertex-disjoint and $C_i$ is not homotopic to $C_j$ for all distinct $i,j \in [h]$, 
%\katie{and $C_i,C_j$ are vertex-disjoint (actually perhaps vertex-disjoint is the necessary condition, and homotopy is irrelevant?)} 
%\tonyh{I think we need both conditions.  For example, if we have a sequence of vertex-disjoint and pairwise homotopic cycles $D_1, \dots, D_k$ ordered as in the proof, then many of the cylinders will intersect.  For example, the cylinder $\Sigma(D_1, D_k)$ intersects $\Sigma(D_i, D_j)$ for all $i,j$. For clarity, I added vertex-disjoint above.}
we note that $\Sigma(C_i, C_i')$ and $\Sigma(C_j, C_j')$ are disjoint for all distinct $i,j \in [h]$.  
We choose a collection of cycles in $G$ such that $\bigcup_{i \in [h]} \Sigma(C_i, C_i')$ is inclusionwise minimal.   We claim that no $\Sigma(C_i, C_i')$ contains a vertex of $G$. Suppose for contradiction that this is not true. Then since $(G, \psi)$ is a triangulation, there must be some $i \in [h]$ and a triangle $T=xyz$ of $G$ such that $\psi(T)$ is the boundary of a face of $(G, \psi)$, with $xy \in E(C_i)$ and $\psi(z) \in \Sigma(C_i, C_i')$. Thus, replacing $C_i$ by $(C_i \cup \{xz, yz\}) \setminus \{xy\}$ contradicts the choice of $C_1, C_1', \dots, C_h, C_h'$.  

To complete the proof we let $\mathbb{M}':=\mathbb{M}  \setminus \bigcup_{i\in[h]} \Sigma(C_i, C'_i) $ and note that $\mathbb{M}'$ is a sphere with $2h$ holes.   
Thus, $H:=\psi^{-1}(\psi(G) \cap \mathbb{M'})$ is a spanning subgraph of $G$ such that $(H,\psi|_H)$ is a triangulated $\surface_{2h}$, as required.  
\end{proof}
 
We remark that~\Cref{lem:homtopiccyles} (with a noncomputable facewidth bound) is a special case of a much more general result by Robertson and Seymour~\cite{RS1988};   
who proved that for every fixed graph $H$ embedded on a surface $\mathbb{M}$ of Euler genus at least 1, there exists a constant $c(H)$ such that every graph $G$ embedded on $\mathbb{M}$ with facewidth at least $c(H)$ contains $H$ as a `surface minor'.  Roughly speaking, the surface minor relation is a refinement of the usual minor relation on graphs, where edge contractions and deletions are performed on $\mathbb{M}$. 
In contrast, our proof gives an explicit bound, and can be considered as part of an extensive line of work to obtain computable bounds for Robertson and Seymour's graph minors project (see for example~\cite{GHR18} and~\cite{wollan22}). 

Another related theorem of Thomassen~\cite[Theorem 3.3 with $d=1$]{Thomassen1993}  implies that for triangulations $G$ of $\mathbb{T}_h$ with facewidth at least 
$2^{h+4} -16$, there is a collection of $h$ pairwise disjoint and pairwise nonhomotopic cycles in $G$.

%-----------------------------------------------------
\section{Triangulations without spanning spheres with few holes}\label{subsec_Counterexample}
%-----------------------------------------------------
In this section we prove~\Cref{thm:fewerholes}.  Given a triangulation $T$, the \defn{face subdivision} of $T$ is obtained by adding a vertex into each face of $T$ and joining this vertex to each of the three vertices on the boundary of its corresponding face.

\begin{thm} \label{thm:facesubdivision}
     Let $\mathbb{M}$ be a closed surface with Euler genus $g \geq 1$ , $T$ be a triangulation of $\mathbb{M}$, and $(G,\psi)$ be the face subdivision of $T$.  Then for every $g' \leq g-1$, $(G,\psi)$ does not contain a spanning subgraph $(H,\psi|_H)$ which is a triangulation of $\surface_{g'}$.
\end{thm}

\begin{proof}
    We colour the vertices of $T$ black and colour the vertices in $V(G) \setminus V(T)$ white.  Let $B$ and $W$ be the set of black and white vertices respectively.  
    By an application of Euler's formula to $T$, we have $|W|= 2|B| +2g-4$, since $|W|$ is equal to the number of faces of $T$. 

    Towards a contradiction, suppose that $G$ has a spanning subgraph $H$ such that $(H,\psi|_H)$ is a triangulated $\surface_{g'}$ for some $g' \leq g-1$.  We choose $H$ to first be maximal under subgraph inclusion, and subject to that, then minimal with respect to $g'$.    
    Let $C_1, \dots , C_{g'}$ be the cycles in $H$ which are the boundaries of the holes of $\surface_{g'}$. The minimality of $g'$ implies that $|V(C_i)| \geq 4$ for all $i \in [g']$.  Let $X:= \bigcup_{i=1}^{g'} V(C_i)$, $W_1:=W \setminus X$ and $W_2:=W \cap X$. Since every vertex contained in $W_1$ is an interior vertex for the triangulation, $\deg_H(w) =3$ for all $w\in W_1$.

    Since $(H,\psi)$ is a triangulation, we have that $\deg_H(v) \geq 2$ for all $v \in V(H)$.
    We claim that $\deg_H(w) =2$ for all $w\in W_2$.  Towards a contradiction suppose $w \in V(C_j)$ for $1\leq j \leq g'$ is a white vertex and $\deg_H(w)=3$.  Let $x,y,z$ be the neighbors of $w$.  Since $w \in X$ and $\deg_H(w)=3$, exactly two of the three faces incident to $w$ are contained in $\surface_{g'}$. By symmetry, we may assume that the faces bounded by $wxz$ and $wyz$ are contained in $\surface_{g'}$, but the face bounded by $wxy$ is not contained in $\surface_{g'}$. Thus, $x$ and $y$ are the neighbors of $w$ on $V(C_j)$.  However, since $|V(C_j)| \geq 4$, $H \cup \{xy\}$ contradicts the maximality of $H$. Therefore, all white vertices in $W_2$ have degree 2 in $H$, as claimed.
   
    We now consider the induced subgraph $H' = H[B\cup W_2]$. 
    Observe that $(H',\psi|_{H'})$ is formed from $(H,\psi|_H)$ by deleting each vertex $w\in W_1$, and merging its 3 incident triangular faces into a single triangular face $F_w$. Since every face of $(H,\psi|_H)$ is incident to exactly one white vertex, these operations are disjoint and leave the faces incident to vertices of $W_2$ untouched. Thus $(H',\psi|_{H'})$ is a triangulation of $\surface_{g'}$ with exactly $|W_1|+|W_2| =|W|$ faces.
    
    Let $D_1, \dots, D_{g'}$ be the cycles of $H'$ which are the boundaries of the holes of $\surface_{g'}$, and let $\ell:=\sum_{i \in [g']} |V(D_i)|$. Consider $(H',\psi|_{H'})$ as embedded on the sphere $\overline{\surface_{g'}}$ with $|W|+ g'$ faces. Since each of the $|W|$ faces of $(H',\psi|_{H'})$ is triangular, we have $2|E(H')|=3|W|+\ell$.  Applying Euler's formula for planar embeddings gives $|W|=2|V(H')|+2g'-4 - \ell$.
    Since the white vertices are a stable set of $G$, we conclude that $W_2$ is also a stable set of $\bigcup_{i \in [g']} D_i$.  Thus, $|W_2| \leq \frac{1}{2} \sum_{i \in [g']} |V(D_i)|=\frac{\ell}{2}$.  Combining everything yields,
    \[
    |W|=2|V(H')|+2g'-4 - \ell  \leq 2|B|+2g'-4 < 2|B|+2g-4 = |W|,
    \]
    which is a contradiction. \qedhere 
\end{proof}

Note that~\Cref{thm:facesubdivision} implies~\Cref{thm:fewerholes} by taking $T$ to be a high facewidth triangulation in~\Cref{thm:facesubdivision}.

\section*{Acknowledgements}

This project originated from the Fields Institute Focus Program on Geometric Constraint Systems.
The authors are grateful to the Fields Institute for their hospitality and financial support. 
K.\,C.\ was supported by the Australian Government through the Australian Research
Council’s Discovery Projects funding scheme (project DP210103849). 
S.\,D.\ was supported by the Heilbronn Institute for Mathematical Research.
M.\,G.\ and T.\,H.\ were supported by the Institute for Basic Science (IBS-R029-C1).
A.\,N.\ was partially supported by EPSRC grant EP/X036723/1.
N.\,F.\ and T.\,H.\ thank the Institute for Basic Science in South Korea where part of this research was conducted.

\bibliographystyle{plainurl}
\bibliography{references}

\end{document}